\documentclass[12pt]{article}
\usepackage{amsfonts, amsmath, amsthm, amssymb, latexsym, amscd, color}
\usepackage{mathrsfs,calrsfs}
\usepackage{graphicx,tikz}
\newcommand{\comment}[1]{}
\newif\ifpdf
\ifpdf \else \fi \textwidth = 6.5 in \textheight = 9 in
\newtheorem{thm}{Theorem}[section]

\newtheorem{observation}[thm]{Observation}

\newtheorem{example}[thm]{Example}
\newtheorem{corollary}[thm]{Corollary}
\newtheorem{lemma}[thm]{Lemma}
\newtheorem{definition}[thm]{Definition}
\newtheorem{theorem}[thm]{Theorem}

\begin{document}

\title{On the distance domination number of bipartite graphs}
\author{{\small D.A. Mojdeh$^{a}$, S.R. Musawi}$^{b}$ {\small , E. Nazari$^{c}$}\\{\small $^{a}$Department of
Mathematics, University of Mazandaran,}{\small Babolsar, Iran}\\ {\small 
damojdeh@umz.ac.ir} \\{\small $^{b}$
\small Department of Mathematics, University of Tafresh,}{\small Tafresh, Iran}\\{\small r\_ musawi@yahoo.com}
\\{\small $^{c}$\small Department of Mathematics, University of Tafresh,}{\small Tafresh, Iran}\\{\small nazari.esmaeil@gmail.com}
}
\date{}
\date{}
\maketitle

\begin{abstract}

A subset $D\subseteq V(G)$ is called a $k$-distance dominating set of G if every vertex in $V (G)-D$ is within distance $k$ from some vertex of $D$. The minimum cardinality among all $k$-distance dominating sets of $G$ is called the $k$-distance domination number of $G$. In this note we give upper bounds on the $k$-distance domination number of a connected bipartite graph, and improve some  results have been given like Theorems 2.1 and 2.7 in [Tian and Xu, A note on distance domination of graphs, Australasian Journal of Combinatorics, 43 (2009), 181-190].
\end{abstract}

\section{Introduction}

For terminology and notation on graph theory not given here, the
reader is referred to West \cite{w}. Let $G$ be a simple graph
with vertex set $V=V(G)$ and edge set $E=E(G)$. The order of $G$
is denoted by $n=|V(G)|=|V|$ and the size of $G$ is denoted by
$m=|E(G)|=|E|$. The \textit{open neighborhood} of a vertex $v$;
$N(v)$ is the set $\{u \in V:\ uv\in E\}$, and the \textit{closed
neighborhood} of a vertex $v$ is the set $N[v] = N(v) \cup
\{v\}$. For a vertex $v\in V$, the \textit{degree} of $v$ is
$\deg_G(v)=\deg(v)=|N(v)|$. The \textit{open
neighborhood} of a set $S\subseteq V$ is the set $N(S)= \cup_{v\in S} N(v) $%
, and the \textit{closed neighborhood} of $S$ is the set $N[S]=N(S)\cup S$. The
\textit{minimum degree} and \textit{maximum degree} of a graph $G$ are denoted by $%
\delta=\delta(G)$ and $\Delta=\Delta(G)$, respectively.  The \textit{open
$k$-neighborhood }of a vertex $v\in V$, denoted $N_{k}(v)$, is the
set $N_{k}(v)=\{u: u\neq v    \ \textit{and} \ d(u,v)\leq k\}$, in the other words $N_{k}(v)$ is the set
of vertices in within distance $k$ of vertex $v$. The
set $N_{k}[v]= N_{k}(v) \cup \{v\}$ is said to be the\textit{ closed
$k$-neighborhood} of $v$.\\
A set $S\subseteq V$ is a \textit{dominating set}
if every vertex in $V$ is either in $S$ or is adjacent to a vertex
in $S$. The \textit{domination number} $\gamma (G)$ is the minimum
cardinality of a dominating set of $G$.
A subset $S\subseteq V$ is a \textit{$k$-distance dominating set} if every
vertex in $V-S$ is within distance $k$ of at least one vertex in
$S$. In the other words, if $S\subseteq V$ is a
$k$-distance dominating set of $G$, then $N_{k}[S]=V$. The $k$-distance
domination number $\gamma_{k}(G)$ of $G$ equals the minimum
cardinality of a $k$-distance dominating set in $G$, for further see,  \cite{hmv, hhs1, msaz}.
The $k$th power graph of $G$ is the
graph with vertex set $V(G)$ and two vertices are adjacent in $G^k$ if they are
joined in $G$ by a path of length at most $k$. Note that $\gamma _{k}(G)$ equals to $\gamma (G^{k})$, where $G^{k}$ is the $k$th power graph of $G$, see \cite{hhs1, tx}.

%
\section{Previous known results}

Tian and Xu \cite{tx} studied $k$-distance domination number in graphs. They have proved the following results.
\begin{theorem}[Tian and Xu \cite{tx}, Theorem 2.1]\label{thm2.1} Let $G$ be a connected graph with vertex set $V=\{1,2,\cdots,n\}$. Then $$\gamma_k(G)\le \underset{(p_1,p_2,\cdots,p_n)\in (0,1)^n}{\min}\underset{i=1}{\overset{n}{  \sum}}\Big( p_i+(1-p_i)\underset{j\in N_k(i)}{\Pi}(1-p_j)\Big)$$ where $p_i\in (0,1)$ is the probability of existence of the vertex $i$ in a random subset of $V$.
\end{theorem}
Then they considered connected bipartite graph.
\begin{lemma}[Tian and Xu \cite{tx}, Lemma 2.5]\label{le1} Let $G$ be a connected bipartite graph
with bipartition $V_1$ and $V_2$, where $|V_j | = n_j$ and
$\delta_j=\min \{\deg(v): v\in V_j\}$, for $j=1,2$.\\
For any vertex $v\in V_{1}$ with $N_k[v]\neq V$,
\begin{equation}\label{eq1}|N_{k}(v)\cap V_{1}|\geq (\lceil k/6 \rceil-1)(\delta_2+1),\end{equation}
\begin{equation}\label{eq2}|N_{k}(v)\cap V_{2}|\geq \lceil k/6 \rceil(\delta_1+1)-1.\end{equation}
Similarly, for any vertex $v\in V_2$ with $N_k[v]\neq V$,
\begin{equation}\label{eq3}|N_{k}(v)\cap V_{1}|\geq   \lceil k/6 \rceil(\delta_2+1)-1,  \end{equation}
\begin{equation}\label{eq4}|N_{k}(v)\cap V_{2}|\geq  (\lceil k/6 \rceil-1)(\delta_1+1). \end{equation}
\end{lemma}

A connected bipartite graph $G$ is said to be \textit{perfect} if $\delta_1\delta_2>1$, $n_1[M(\delta_1+1)-1]>n_2[(M-1)(\delta_2+1)+1]$ and $n_2[M(\delta_2+1)-1]>n_1[(M-1)(\delta_1+1)+1]$, where $M=\lceil k/6 \rceil$. A simple calculation shows that a connected bipartite graph is perfect if and only if $n_1-n_2\delta_2< M[n_1(\delta_1+1)-n_2 (\delta_2+1)]<n_1 \delta_1-n_2$. As a consequence of Lemma \ref{le1} and Theorem \ref{thm2.1}, Tian and Xu obtained the following.

\begin{theorem}[Tian and Xu \cite{tx}, Theorem 2.7]\label{thm2.7}
Let $G$ be a perfect bipartite graph and $$0<p_1=\frac{[(M-1)(\delta_1+1)+1]\ln u-[M(\delta_1+1)-1]\ln v}{(2M-1)(\delta_1\delta_2-1)}<1$$
 $$ 0<p_2=\frac{[(M-1)(\delta_2+1)+1]\ln v-[M(\delta_2+1)-1]\ln u}{(2M-1)(\delta_1\delta_2-1)}<1,$$
where $u=\frac{n_2[M(\delta_2+1)-1]-n_1[(M-1)(\delta_1+1)+1]}{n_1(2M-1)(\delta_1\delta_2-1)}$ and  $v=\frac{n_1[M(\delta_1+1)-1]-n_2[(M-1)(\delta_2+1)+1]}{n_2(2M-1)(\delta_1\delta_2-1)}.$ Then
$$\gamma_k(G)\le h(p_1,p_2)\le \underset{0< p <1}{\min}h(p,p) \le \frac{n(1+\ln[(2M-1)(\delta+1)])}{(2M-1)(\delta+1)},$$
where $M=\lceil k/6\rceil$.
\end{theorem}

In this manuscript we improve Theorem \ref{thm2.7} via improving the Lemma \ref{le1}.

%
\section{Main results}

In order to improve Theorem \ref{thm2.7}, we first improve Lemma \ref{le1}.

\begin{lemma}\label{m1}
Let $G$ be a connected bipartite graph with bipartition $V_1$ and $V_2$, where $|V_j | = n_j$ and $\delta_j=\min \{\deg(v): v\in V_j\}$, for $j=1,2$. Then\\
$(i)$ For any vertex $v\in V_{1}$ with $N_k[v]\neq V$,
\begin{equation}\label{eq5}
|N_{k}(v)\cap V_{1}|\geq \lceil (k-1)/4 \rceil \max\{2,\delta_2\}+2\lfloor k/4\rfloor -\lfloor k/2\rfloor,
\end{equation}
\begin{equation}\label{eq6}
|N_{k}(v)\cap V_{2}|\geq \delta_1+(\lceil k/4 \rceil -1) \max\{2,\delta_1\}+\lfloor(k-1)/2\rfloor -2\lfloor (k-1)/4\rfloor.
\end{equation}
Furthermore, (\ref{eq5}) and (\ref{eq6}), improve (\ref{eq1}) and (\ref{eq2}), repectively.\\
$(ii)$ For any vertex $v\in V_2$ with $N_k[v]\neq V$,
\begin{equation}\label{eq7}
|N_{k}(v)\cap V_{1}|\geq  \lceil k/4 \rceil \max\{2,\delta_2\}+\lfloor(k-1)/2\rfloor -2\lfloor (k-1)/4\rfloor,
\end{equation}
\begin{equation}\label{eq8}
|N_{k}(v)\cap V_{2}|\geq  \lceil (k-1)/4 \rceil \max\{2,\delta_1\}+2\lfloor k/4\rfloor -\lfloor k/2\rfloor.
\end{equation}
Furthermore, (\ref{eq7}) and (\ref{eq8}) improve (\ref{eq3}) and (\ref{eq4}), respectively.
\end{lemma}

\begin{proof}
Let $G$ be a connected bipartite graph with bipartition $V_1$ and
$V_2$, where $|V_j | = n_j$ and $\delta_j=\min \{\deg(v): v\in
V_j\}$, for $j=1,2$. For any vertex $v$ and any integer $l$ with
$1\leq l \leq k$, let $X_l(v)=\{u\in V | d(v,u)=l\}$. It is obvious that
$N_{k}(v)=X_1(v)\cup X_2(v)\cup
\dots\cup X_k(v).$ Furthermore, $X_1(v)$, $X_2(v)$,...,and $\dots, X_k(v)$ are pairly disjoint.

$(i)$ Let $v\in V_{1}$ be a vertex with $N_k[v]\neq V$. Observe
that $ X_1(v)\cup X_3(v)\cup \dots\cup X_{2\lfloor
(k+1)/2\rfloor-1}(v) \subseteq V_2$, $X_2(v)\cup X_4(v)\cup
\dots\cup X_{2\lfloor k/2\rfloor}(v) \subseteq V_1$, and
\begin{equation*}\label{eq9}
N_{k}(v)\cap V_{1}= \underset{m=1}{\overset{\lfloor k/2\rfloor
}{\bigcup}}X_{2m}(v),\ \ N_{k}(v)\cap V_{2}=
\underset{m=1}{\overset{\lfloor (k+1)/2\rfloor }\bigcup
}X_{2m-1}(v)
\end{equation*}
Thus $|N_{k}(v)\cap V_{1}|= \underset{m=1}{\overset{\lfloor
k/2\rfloor }{\sum}}|X_{2m}(v)|$ and $| N_{k}(v)\cap V_{2}|=
\underset{m=1}{\overset{\lfloor (k+1)/2\rfloor
}{\sum}}|X_{2m-1}(v)|$. Since $N_k[v]\neq V$, there exists a
vertex $u$ such that $d(v,u)>k$. Then there exists a path,
$P:=vx_1x_2\dots u$ of length of at least $k+1$.  For $l=1,2,\cdots,k$, $X_{l}(v)\neq
\emptyset$, because $x_l\in X_l(v)$. Moreover, if $l$
is odd, then $\deg(x_l)\ge \max\{2,\delta_2 \}$, because $x_l\in V_2$; while if
$l$ is even, then $\deg(x_l)\ge \max\{2,\delta_1\}$,  because $x_l\in V_1$. We proceed with the following claims.

\textbf{Claim} 1. $|X_2(v)|\geq \max\{2,\delta_2\}-1 \geq \delta_2-1.$\\
To see this, note that since $x_1\in X_1(v)\subseteq V_2$, we have
$|X_2(v)|=\deg(x_1)-1$. Since $\deg(x_1)\ge \max\{2,\delta_1\}$, we find that $|X_2(v)|\geq max\{2,\delta_2\}-1$, as desired.

\textbf{Claim} 2. For $2\leq l\leq k-1$,  $|X_{l-1}(v)|+|X_{l+1}(v)|\geq \deg(x_l)$.\\
To see this, note that for $2\leq l\leq k-1$, we have $N_1(x_l)=N(x_l)\subseteq  X_{l-1}(v)\cup  X_{l+1}(v)$, since $x_l\in X_l(v)$.

By Claim 2, $|X_{4m}(v)|+|X_{4m+2}(v)|\geq \deg(x_{4m+1})$ for every $m=1,2,...,\lfloor\frac{\lfloor k/2\rfloor-1}{2}\rfloor$. To compute $|N_{k}(v)\cap V_{1}|$, we discuss  on $\frac{\lfloor k/2\rfloor-1}{2}$ which may be an  integer or not.

First we assume that $\frac{\lfloor k/2\rfloor-1}{2}$ is an integer. Then
 \begin{align*}
|N_{k}(v)\cap V_{1}|&= \underset{m=1}{\overset{\lfloor k/2\rfloor }{\sum}}|X_{2m}(v)| \\
& =  |X_2(v)|+\underset{m=2}{\overset{\lfloor k/2\rfloor }{\sum}}|X_{2m}(v)|\\&=  |X_2(v)|+ \underset{m'=1}{\overset {(\lfloor k/2\rfloor-1)/2 }{\sum}}(|X_{4m'}(v)|+|X_{4m'+2}(v)|)\\
&\geq \max\{2,\delta_2\}-1+ \underset{m'=1}{\overset {(\lfloor
k/2\rfloor-1)/2 }{\sum}}\max\{2,\delta_2\} \ \ \ \ \  (by \   Claims \ 1 \ and \ 2).
\end{align*}
Thus $|N_{k}(v)\cap V_{1}|\geq (\lfloor k/2\rfloor+1)\max\{2,\delta_2\}/2-1$ and a simple calculation shows that $(\lfloor k/2\rfloor+1)\max\{2,\delta_2\}/2-1=\lceil (k-1)/4 \rceil \max\{2,\delta_2\}+2\lfloor k/4\rfloor -\lfloor k/2\rfloor$, as desired.

Next we assume that $\frac{\lfloor k/2\rfloor-1}{2}$ is not an integer. Then
\begin{align*}
|N_{k}(v)\cap V_{1}|&= \underset{m=1}{\overset{\lfloor k/2\rfloor }{\sum}}|X_{2m}(v)| \\
& =  |X_2(v)|+\underset{m=2}{\overset{\lfloor k/2\rfloor-1 }{\sum}}|X_{2m}(v)|+|X_{2\lfloor k/2\rfloor}|\\&=  |X_2(v)|+ \underset{m'=1}{\overset {(\lfloor k/2\rfloor-2)/2 }{\sum}}(|X_{4m'}(v)|+|X_{4m'+2}(v)|)+|X_{2\lfloor k/2\rfloor}|\\
&\geq \max\{2,\delta_2\}-1+ \underset{m'=1}{\overset {(\lfloor
k/2\rfloor-2)/2 }{\sum}}\max\{2,\delta_2\}+1 \ \ \ \ \  (by \   Claims \ 1 \ and \ 2).
\end{align*}

Thus $|N_{k}(v)\cap V_{1}|\geq \lfloor k/2\rfloor\max\{2,\delta_2\}/2$ and a simple calculation shows that $\lfloor k/2\rfloor\max\{2,\delta_2\}/2= \lceil (k-1)/4 \rceil \delta_2+2\lfloor k/4\rfloor -\lfloor k/2\rfloor$, as desired.

Hence, inequality (\ref{eq5}) holds. We next prove the inequality (\ref{eq6}). Since $\deg(v)\geq \delta_1$ and $N(v)=X_1(v)\subseteq V_2$, we find that $|X_1(v)|\geq \delta_1$.

From Claim 2, we can easily see that $|X_{4m-1}(v)|+|X_{4m+1}(v)|\geq \deg(x_{4m})\ge \max\{2,\delta_1\}$ for every $m=1,2,...,\lfloor\frac{k-1}{4}\rfloor$. We discuss  on $\frac{\lfloor (k+1)/2\rfloor}{2}$ which may be an  integer or not.  

First we assume that $\frac{\lfloor (k+1)/2\rfloor}{2}$ is an integer. Then
\begin{align*}
|N_{k}(v)\cap V_{2}|&= \underset{m=1}{\overset{\lfloor (k+1)/2\rfloor }{\sum}}|X_{2m-1}(v)| \\
& =  |X_1(v)|+\underset{m=2}{\overset{\lfloor (k+1)/2\rfloor }{\sum}}|X_{2m-1}(v)|\\
&=  |X_1(v)|+ \underset{m'=1}{\overset {\lfloor (k+1)/2\rfloor/2-1  }{\sum}}(|X_{4m'-1}(v)|+|X_{4m'+1}(v)|)+|X_{2\lfloor (k+1)/2\rfloor -1}(v)|\\
&\geq \delta_1+ \underset{m'=1}{\overset {\lfloor
(k+1)/4\rfloor-1 }{\sum}}\max\{2,\delta_1\}+1 \ \ \ \ \  (by \   Claim \ 2).
\end{align*}
Thus $|N_{k}(v)\cap V_{2}| \geq \delta_1+(\lfloor(k+1)/4\rfloor-1)\max\{2,\delta_1\}+1$. Now a simple calculation shows that $ \delta_1+(\lfloor(k+1)/4\rfloor-1)\max\{2,\delta_1\}+1=\delta_1+(\lceil k/4 \rceil -1) \max\{2,\delta_1\}+\lfloor(k-1)/2\rfloor -2\lfloor (k-1)/4\rfloor$ as desired.

Next we assume that $\frac{\lfloor (k+1)/2\rfloor}{2}$ is not an integer. Then
\begin{align*}
|N_{k}(v)\cap V_{2}|&= \underset{m=1}{\overset{\lfloor (k+1)/2\rfloor }{\sum}}|X_{2m-1}(v)| \\
& =  |X_1(v)|+\underset{m=2}{\overset{\lfloor (k+1)/2\rfloor }{\sum}}|X_{2m-1}(v)|\\
&=  |X_1(v)|+ \underset{m'=1}{\overset {(\lfloor (k+1)/2\rfloor-1)/2  }{\sum}}(|X_{4m'-1}(v)|+|X_{4m'+1}(v)|)\\
&\geq \delta_1+ \underset{m'=1}{\overset {\lfloor (k-1)/4\rfloor
}{\sum}}\max\{2,\delta_1\} \ \ \ \ \  (by \   Claim \ 2).
\end{align*}
Thus $|N_{k}(v)\cap V_{2}| \geq \delta_1+\lfloor (k-1)/4\rfloor\max\{2,\delta_1\}$. Now a simple calculation shows that $ \delta_1+\lfloor(k-1)/4\rfloor\max\{2,\delta_1\}=\delta_1+(\lceil k/4 \rceil -1) \max\{2,\delta_1\}+\lfloor(k-1)/2\rfloor -2\lfloor (k-1)/4\rfloor$ as desired.

We next show that inequality \ref{eq5} is an improvement of inequality \ref{eq1}. We will show that:
\[  \lceil \frac{k-1}{4}\rceil \max\{2,\delta_2\}+2\lfloor \frac{k}{4}\rfloor -\lfloor \frac{k}{2}\rfloor\geq (\lceil \frac{k}{6} \rceil-1)(\delta_2+1)  \]
It is obvious that if $\delta_2=1$, then the left side of the above inequality is $2\lceil \frac{k-1}{4}\rceil+2\lfloor \frac{k}{4}\rfloor -\lfloor \frac{k}{2}\rfloor=\lfloor \frac{k}{2}\rfloor$ and the right side is $2(\lceil \frac{k}{6} \rceil-1)$, and clearly $2\lceil \frac{k-1}{4}\rceil+2\lfloor \frac{k}{4}\rfloor -\lfloor \frac{k}{2}\rfloor=\lfloor \frac{k}{2}\rfloor\geq 2(\lceil \frac{k}{6} \rceil-1)$ for $k\ge 1$. Thus assume that $\delta_2\ge2$. We show that
\[ ( \lceil \frac{k-1}{4}\rceil-\lceil \frac{k}{6} \rceil+1) \delta_2 \geq \lceil \frac{k}{6} \rceil-1 -2\lfloor \frac{k}{4}\rfloor +\lfloor \frac{k}{2}\rfloor \]
for $k\geq1$.
Let $L=( \lceil \frac{k-1}{4}\rceil-\lceil \frac{k}{6} \rceil+1) \delta_2 $ and $ R= \lceil \frac{k}{6} \rceil-1 -2\lfloor \frac{k}{4}\rfloor +\lfloor \frac{k}{2}\rfloor$. We thus show that $L\geq R$. Let $k=12p+q$, where $1\leq q \leq 12$. Then
\[  L=( \lceil \frac{k-1}{4}\rceil-\lceil \frac{k}{6} \rceil+1) \delta_2= p \delta_2+(\lceil \frac{q-1}{4}\rceil-\lceil \frac{q}{6} \rceil+1) \delta_2.    \]
\[  R= \lceil \frac{k}{6} \rceil-1 -2\lfloor \frac{k}{4}\rfloor +\lfloor \frac{k}{2}\rfloor =2p+ \lceil \frac{q}{6} \rceil-1 -2\lfloor \frac{q}{4}\rfloor +\lfloor \frac{q}{2}\rfloor.   \]
Since $\delta_2\geq 2$, we have $p\delta_2\geq 2p$. Thus we need to show that $(\lceil \frac{q-1}{4}\rceil-\lceil \frac{q}{6} \rceil+1) \delta_2\geq \lceil \frac{q}{6} \rceil-1 -2\lfloor \frac{q}{4}\rfloor +\lfloor \frac{q}{2}\rfloor.$ Since $1\leq q \leq 12$ we show this by Table 1.

\begin{center}
Table 1\\
\begin{tabular}{|c||c|c|c|c|c|c|c|c|c|c|c|c|}
\hline
 $q$ &1&2&3&4&5&6&7&8&9&10&11&12
\\
\hline $(\lceil \frac{q-1}{4}\rceil-\lceil \frac{q}{6} \rceil+1)
\delta_2.$ &0&$\delta _2$&$\delta _2$&$\delta _2$&$\delta
_2$&$2\delta _2$&$\delta _2$&$\delta _2$&$\delta _2$&$2\delta
_2$&$2\delta _2$&$2\delta _2$
\\
\hline $\lceil \frac{q}{6} \rceil-1 -2\lfloor \frac{q}{4}\rfloor
+\lfloor \frac{q}{2}\rfloor$ &0&1&1&0&0&1&2&1&1&2&2&1\\ \hline
\end{tabular}
\end{center}

Thus  inequality (\ref{eq5}) is an improvement of inequality (\ref{eq1}). We next show that inequality (\ref{eq6}) is an improvement of inequality (\ref{eq2}).  We will show that :
\[  \delta_1+(\lceil k/4 \rceil -1) \max\{2,\delta_1\}+\lfloor(k-1)/2\rfloor -2\lfloor (k-1)/4\rfloor \ge \lceil k/6 \rceil(\delta_1+1)-1  \]
If $\delta_1=1$, then the above inequality becomes $1+2(\lceil k/4 \rceil -1) +\lfloor(k-1)/2\rfloor -2\lfloor (k-1)/4\rfloor=\lceil k/2 \rceil \ge 2\lceil \frac{k}{6} \rceil-1$ which is valid for any $k\geq 1$. Thus we assume that $\delta_1\ge 2$. It is sufficient to show that
\[( \lceil \frac{k}{4}\rceil-\lceil \frac{k}{6} \rceil) \delta_1 \geq \lceil \frac{k}{6} \rceil-1 -\lfloor \frac{k-1}{2}\rfloor +2\lfloor \frac{k-1}{4}\rfloor \] for $k\geq 1$. Let  $L=( \lceil \frac{k}{4}\rceil-\lceil \frac{k}{6} \rceil) \delta_1$ and $ R=\lceil \frac{k}{6} \rceil-1 -\lfloor \frac{k-1}{2}\rfloor +2\lfloor \frac{k-1}{4}\rfloor$. We thus need to show that $L\geq R$. Let $k=12p+q$, where $1\leq q \leq 12$. Then
\[  L=( \lceil \frac{k}{4}\rceil-\lceil \frac{k}{6} \rceil) \delta_1= p \delta_1+( \lceil \frac{q}{4}\rceil-\lceil \frac{q}{6} \rceil) \delta_1.    \]
\[  R=\lceil \frac{k}{6} \rceil-1 -\lfloor \frac{k-1}{2}\rfloor +2\lfloor \frac{k-1}{4}\rfloor =2p+\lceil \frac{q}{6} \rceil-1 -\lfloor \frac{q-1}{2}\rfloor +2\lfloor \frac{q-1}{4}\rfloor. \]
Since $\delta_1\geq 2$, we have $p\delta_1\ge 2p$. Thus we it is sufficient to show that $ ( \lceil \frac{q}{4}\rceil-\lceil \frac{q}{6} \rceil) \delta_1\ge \lceil \frac{q}{6} \rceil-1 -\lfloor \frac{q-1}{2}\rfloor +2\lfloor \frac{q-1}{4}\rfloor.$ We do this in Table 2, since $1\leq q \leq 12$.
\begin{center}
Table 2\\
\begin{tabular}{|c||c|c|c|c|c|c|c|c|c|c|c|c|}
\hline $q$ &1&2&3&4&5&6&7&8&9&10&11&12
\\
\hline
$( \lceil \frac{q}{4}\rceil-\lceil \frac{q}{6} \rceil) \delta_1$ &0&0&0&0&$\delta_1$&$\delta_1$&0&0&$\delta_1$&$\delta_1$&$\delta_1$&$\delta_1$
\\
\hline
$\lceil \frac{q}{6} \rceil-1 -\lfloor \frac{q-1}{2}\rfloor +2\lfloor \frac{q-1}{4}\rfloor$ &0&0&-1&-1&0&0&0&0&1&1&0&0
\\
\hline
\end{tabular}
\end{center}

Thus (\ref{eq6}) is an improvement of (\ref{eq2}).

The proof of part ($ii$), (i.e.  (\ref{eq7}) and
(\ref{eq8})) is similar and straightforward, and therefore is
omitted.
\end{proof}

\begin{theorem}
If $G$ is a bipartite  graph and $k$ is a positive integer, then
 $$ \gamma_k(G)\leq \underset{(p_1,p_2)\in (0,1)^2} \min h^*(p_1,p_2),$$ where
 \begin{align*}
  h^*(p_1,p_2)&=n_1p_1+n_1e^{-p_1(A_{11}+1)-p_2A_{12}} + n_2p_2+n_2e^{-p_1A_{21}-p_2(A_{22}+1)},   \\
A_{11}&=\lceil (k-1)/4 \rceil \max\{2,\delta_2\}+2\lfloor k/4\rfloor -\lfloor k/2\rfloor  \\
A_{12}&=\delta_1+(\lceil k/4 \rceil -1) \max\{2,\delta_1\}+\lfloor(k-1)/2\rfloor -2\lfloor (k-1)/4\rfloor  \\
A_{21}&=\delta_2+(\lceil k/4 \rceil -1) \max\{2,\delta_2\}+\lfloor(k-1)/2\rfloor -2\lfloor (k-1)/4\rfloor \\
A_{22}&=\lceil (k-1)/4 \rceil \max\{2,\delta_1\}+2\lfloor k/4\rfloor -\lfloor k/2\rfloor
\end{align*}

This bound improve the bound given in Theorem \ref{thm2.7}.
\end{theorem}
\begin{proof}
By Theorem \ref{thm2.1}, we have
\begin{align*}
\gamma_k(G)\leq \min_{(p_1,p_2)\in (0,1)^2}\bigg(&\sum_{v\in V_1}\big [p_1+(1-p_1)^{|N_k(v)\cap V_1|+1}(1-p_2)^{|N_k(v)\cap V_2|}\big ] \\
&\qquad +\sum_{v\in V_2}\big [p_2+(1-p_1)^{|N_k(v)\cap V_1|}(1-p_2)^{|N_k(v)\cap V_2|+1}\big ] \bigg).\end{align*}

By Lemma \ref{m1}, we have
\begin{align*}
\gamma_k(G)&\leq \min_{(p_1,p_2)\in (0,1)^2}\bigg(\sum_{v\in V_1}\big [p_1+(1-p_1)^{A_{11}+1}(1-p_2)^{A_{12}}\big ] \\
&\qquad \qquad \qquad \qquad +\sum_{v\in V_2}\big [p_2+(1-p_1)^{A_{21}}(1-p_2)^{A_{22}+1}\big ] \bigg)  \\
&\le\min_{(p_1,p_2)\in (0,1)^2}\bigg(\big [n_1p_1+n_1(1-p_1)^{A_{11}+1}(1-p_2)^{A_{12}}\big ] \\
&\qquad \qquad \qquad \qquad +\big [n_2p_2+n_2(1-p_1)^{A_{21}}(1-p_2)^{A_{22}+1}\big ] \bigg) \\
&\le\min_{(p_1,p_2)\in (0,1)^2}\bigg( n_1p_1+n_1e^{-p_1(A_{11}+1)-p_2A_{12}} + n_2p_2+n_2e^{-p_1A_{21}-
p_2(A_{22}+1)} \bigg).
\end{align*}

That is $ \gamma_k(G)\leq \underset{(p_1,p_2)\in (0,1)^2} \min h^*(p_1,p_2).  $ To show that our bound is an improvement of the bound given in Theorem \ref{thm2.7}, note that by Lemma \ref{m1} one can easily see that $h^*(p_1,p_2)\leq h(p_1,p_2)$, since $\exp(-x)$ is a decreasing function.
\end{proof}
\begin{example}\label{example} It remains to show that there are perfect graphs that our bound is better than the older one. For this purpose, let $G$ be a connected bipartite graph with $ n_1=n_2=\frac{n}{2}$, $\delta_1=\delta_2=\delta\ge 2$, and $k=4m+1$ with $m=1,2,3,\cdots$. We can easily see that the graph is perfect. Now we have $A_{11}=A_{22}=m\delta , A_{12}=A_{21}=(m+1)\delta$ and
$$h^*(p_1,p_2)=\frac{n}{2}[p_1+p_2+e^{-p_1(m\delta+1)-p_2(m+1)\delta} +e^{-p_1(m+1)\delta-p_2(m\delta+1)}].$$
By using of calculus method, we see that the unique minimum of $h^*$ occurs at $$p_1=p_2=\displaystyle \frac{\ln[(2m+1)\delta+1]}{(2m+1)\delta+1},$$ since $0<p_1=p_2<1$, we have $\underset{(p_1,p_2)\in (0,1)^2}{\min} h^*(p_1,p_2)=n(\displaystyle \frac{1+\ln[(2m+1)\delta+1]}{(2m+1)\delta+1})$. By calculus we know that the function $f(x)=\frac{1+\ln x}{x}$ is decreasing on interval $(1,\infty)$ and also we have $(2m+1)\delta+1\ge (2\lceil k/6 \rceil-1)(\delta+1)$,  thus the new bound refinements the bound in Theorem \ref{thm2.7}.
\end{example}
%
%
\subsection{Minimizing $h^*(p_1,p_2)$}

In this part of paper we wish to minimize $h^*(p_1,p_2)$. For this purpose, we consider two different cases and we use calculation  methods.

\subsubsection{ $k$ is even}

In this case we will show that either $h^*$ hasn't local extremum or it has infinitely local minimum on $(0,1)^2$. However $h^*$ has local minimum on closed unit square $[0,1]^2$, thus we extend the domain of $h^*$ into $[0,1]^2$.

Before introducing our main results, we explain an observation in calculus :
\begin{observation} \label{obs} Consider the function $f(x)=\frac{a+\ln x}{x}$ where $x>0$ and $a>0$. $f$ has a unique maximum in $x=e^{1-a}\le e$ thus $f(x)\le f(e^{1-a})=e^{a-1}$. Now, if $a< 1$ then $f(x)< 1$ for all $x>0$.
\end{observation}

Our main result in this states is :

\begin{theorem}
If $k$ is an even integer, $\delta_1,\delta_2 \ge 2$ and  $T=\max\{\frac{nA_{12}}{n_2},\frac{nA_{21}}{n_1}\}$, in each of cases \\
(i) $\frac{nA_{12}}{n_2}=\frac{nA_{21}}{n_1}$
\\ (ii) $\frac{nA_{12}}{n_2}<\frac{nA_{21}}{n_1}$  and   $\frac{1}{A_{21}}\ln \frac{nA_{21}}{n_1}<1$
\\
(iii) $\frac{nA_{12}}{n_2}>\frac{nA_{21}}{n_1}$ and $\frac{1}{A_{12}}\ln \frac{nA_{12}}{n_2}<1$
\\we have  $\underset{(p_1,p_2)\in (0,1)^2} \inf\  h^*(p_1,p_2)=\underset{(p_1,p_2)\in [0,1]^2} \min\  h^*(p_1,p_2)=n(\frac{1+\ln T}{T}).$
\end{theorem}

\begin{proof}We assume that  $k \overset{4} \equiv 0 $, then  \[ A_{11}=k \delta_2/4 , \quad  A_{12}= k \delta_1/4 +1 ,\quad   A_{21}= k \delta_2/4+1, \quad A_{22}=k \delta_1/4 \]
and if $k \overset{4} \equiv 2$, then
\[ A_{11}=(k+2) \delta_2/4-1 , \quad  A_{12}= (k+2) \delta_1/4  ,\quad   A_{21}= (k+2) \delta_2/4, \quad A_{22}= (k+2) \delta_1/4-1 \]
 thus, in both cases we have $A_{11}+1=A_{21}$ and  $A_{22}+1=A_{12}$, and therefore
 \[  h^*(p_1,p_2)=n_1p_1+n_2p_2+ne^{-p_1A_{21}-p_2A_{12}}   \]

To minimize $h^*(p_1,p_2)$, using partial differential, we have
$h^*_{p_1}=n_1-nA_{21}e^{-p_1A_{21}-p_2A_{12}}$ and $h^*_{p_2}=n_2-nA_{12}e^{-p_1A_{21}-p_2A_{12}}$. Then from  $h^*_{p_1}=0$, we obtain that $e^{-p_1A_{21}-p_2A_{12}}=\frac{n_1}{nA_{21}}$, and so $p_1A_{21}+p_2A_{12}=\ln \frac{nA_{21}}{n_1}$. Likewise, from $ h^*_{p_2}=0$, we obtain $p_1A_{21}+p_2A_{12}=\ln \frac{nA_{12}}{n_2}$.

(i) If $\frac {A_{21}}{n_1}=\frac {A_{12}}{n_2}$ then $h^*$ is constant for all $p_1$ and $p_2$ with $p_1A_{21}+p_2A_{12}=\ln \frac{nA_{12}}{n_2}=\ln \frac{nA_{21}}{n_1}$, as the following shows:.
\begin{align*}
h^*(p_1,p_2)&=n_1p_1+n_2p_2+ne^{-p_1A_{21}-p_2A_{12}}\\
&=\frac{n_1}{A_{21}}(p_1A_{21}+p_2A_{12})+ne^{-p_1A_{21}-p_2A_{12}}\\
&=\frac{n_1}{A_{21}}\ln \frac{nA_{21}}{n_1}+ne^{-\ln \frac{nA_{21}}{n_1}}=\frac{n_1}{A_{21}}(1+\ln \frac{nA_{21}}{n_1})\\
\end{align*}

Note that two points $( 0 , \frac{1}{A_{12}}\ln \frac{nA_{12}}{n_2})$ and $( \frac{1}{A_{21}}\ln \frac{nA_{21}}{n_1} , 0 )$ are located on the line $p_1A_{21}+p_2A_{12}=\ln \frac{nA_{12}}{n_2}=\ln \frac{nA_{21}}{n_1}$ and by Observation \ref{obs} we have $0<\min\{\frac{1}{A_{12}}\ln \frac{nA_{12}}{n_2}, \frac{1}{A_{21}}\ln \frac{nA_{21}}{n_1}  \}<1$ because $0<\min\{\ln \frac{n}{n_2}, \ln \frac{n}{n_1}  \}<1$.
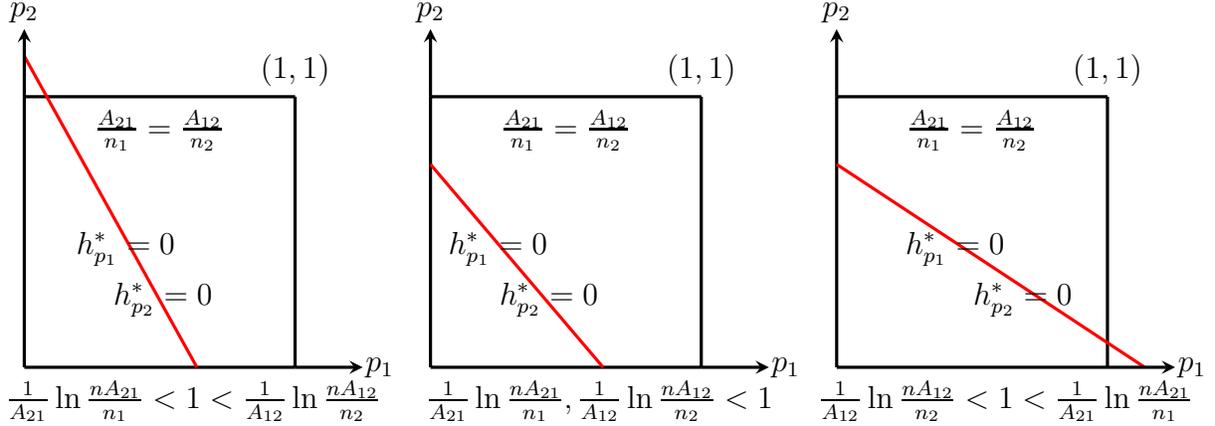
\begin{center}
\begin{tikzpicture}[scale=.9,very  thick]
\draw[ very thick] (-8,2) -- (-8,-2)  (-4,2) -- (-4,-2)   (-8,2) -- (-4,2)   (-8,-2)--(-4,-2);
\draw[very thick,red](-5.45,-2)--(-8,2.6){};
\draw [>=stealth,->]   (-4,-2)-- (-3,-2);
\draw [>=stealth,->]   (-8,2)-- (-8,3);
\node at (-8,3.25) { $p_2$};
\node at (-2.75,-2) { $p_1$};
\node at (-4,2.4) { $(1,1)$};
\node at (-5.95,-1) { $h^*_{p_2}=0$};
\node at (-6.5,-.25) { $h^*_{p_1}=0$};
\node at (-6,1.5) { $\frac{A_{21}}{n_1}=\frac{A_{12}}{n_2}$};
\node at (-5.5,-2.5) { $ \frac{1}{A_{21}}\ln \frac{nA_{21}}{n_1}<1<\frac{1}{A_{12}}\ln \frac{nA_{12}}{n_2}$};

 \draw[ very thick] (-2,2) -- (-2,-2)  (2,2) -- (2,-2)   (-2,2) -- (2,2)   (-2,-2)--(2,-2);
\draw[very thick,red](0.55,-2)--(-2,1){};
\draw [>=stealth,->]   (2,-2)-- (3,-2);
\draw [>=stealth,->]   (-2,2)-- (-2,3);
\node at (-2,3.25) { $p_2$};
\node at (3.25,-2) { $p_1$};
\node at (2,2.4) { $(1,1)$};
\node at (-.25,-1) { $h^*_{p_2}=0$};
\node at (-1,-.25) { $h^*_{p_1}=0$};
\node at (0,1.5) { $\frac{A_{21}}{n_1}=\frac{A_{12}}{n_2}$};
\node at (.5,-2.5) { $ \frac{1}{A_{21}}\ln \frac{nA_{21}}{n_1},\frac{1}{A_{12}}\ln \frac{nA_{12}}{n_2}<1$};
\draw[ very thick] (4,2) -- (4,-2)  (8,2) -- (8,-2)   (4,2) -- (8,2)   (4,-2)--(8,-2);
\draw[very thick,red](8.55,-2)--(4,1){};
\draw [>=stealth,->]   (8,-2)-- (9,-2);
\draw [>=stealth,->]   (4,2)-- (4,3);
\node at (4,3.25) { $p_2$};
\node at (9.25,-2) { $p_1$};
\node at (8,2.4) { $(1,1)$};
\node at (6.75,-1) { $h^*_{p_2}=0$};
\node at (5.75,-.25) { $h^*_{p_1}=0$};
\node at (6,1.5) { $\frac{A_{21}}{n_1}=\frac{A_{12}}{n_2}$};
\node at (6.5,-2.5) { $ \frac{1}{A_{12}}\ln \frac{nA_{12}}{n_2}<1<\frac{1}{A_{21}}\ln \frac{nA_{21}}{n_1}$};

 \end{tikzpicture}

\end{center}

Thus the minimum of $h^*(p_1,p_2)$ is $\frac{n_1}{A_{21}}(1+\ln \frac{nA_{21}}{n_1})$, and note that it happens for every pairs $(p_1,p_2)\in (0,1)^2$, satisfying $h^*_{p_1}=h^*_{p_2}=0$. Now letting $T=\frac {nA_{21}}{n_1}=\frac {nA_{12}}{n_2}$, we obtain that  $\underset{p_1,p_2} \min\  h^*(p_1,p_2)=n(\frac{1+\ln T}{T})$ as desired.

 If $\frac {A_{21}}{n_1}\ne \frac {A_{12}}{n_2}$ then $p_1A_{21}+p_2A_{12}=\ln \frac{nA_{21}}{n_1}$ and $p_1A_{21}+p_2A_{12}=\ln \frac{nA_{12}}{n_2}$ are two distinct parallel lines in the $p_1p_2$-coordinate system. Thus, $ h^*$ has no extremum in $(0,1)^2$ but it has an infimum value in $(0,1)^2$. For this purpose we seek the extremum of $ h^*$ in $[0,1]^2$. Observe that the line $p_1A_{21}+p_2A_{12}=\ln \frac{nA_{21}}{n_1}$ intersects the $p_1$-axis in $M_1=(\frac{1}{A_{21}}\ln \frac{nA_{21}}{n_1},0)$ and $p_2-$axis in $N_1=(0,\frac{1}{A_{12}}\ln \frac{nA_{21}}{n_1})$. Similarly, the line $p_1A_{21}+p_2A_{12}=\ln \frac{nA_{12}}{n_2}$ intersects the $p_1$-axis in $M_2=(\frac{1}{A_{21}}\ln \frac{nA_{12}}{n_2},0)$ and $p_2$-axis in $N_2=(0,\frac{1}{A_{12}}\ln \frac{nA_{12}}{n_2})$. Moreover, let $Q_1=(1,0)$ and $Q_2=(0,1)$.

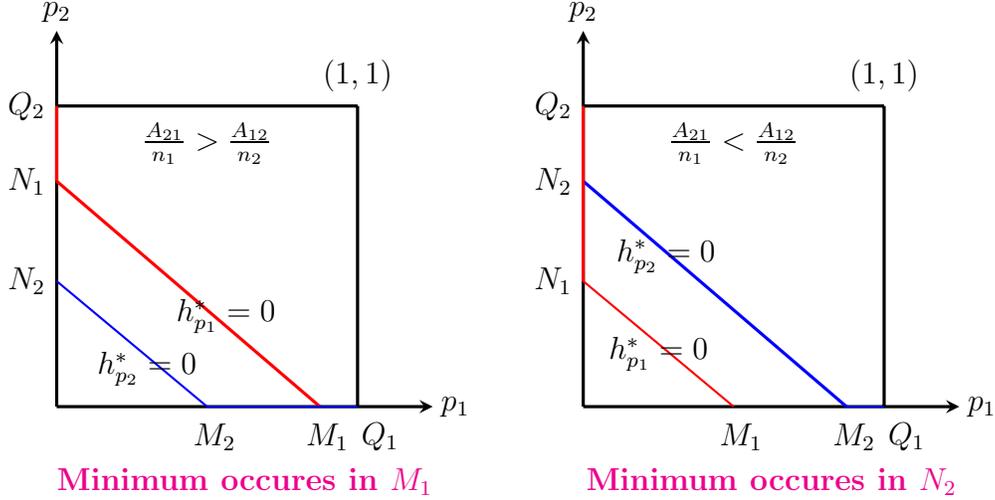
\begin{center}
\begin{tikzpicture}[scale=1,very  thick]

\draw[ very thick] (-2,2) -- (-2,-2)  (2,2) -- (2,-2)   (-2,2) -- (2,2)   (-2,-2)--(2,-2);
\draw[very thick,red](1.5,-2)--(-2,1){};
\draw[very thick,red](-2,2)--(-2,1){};
\draw[ thick,blue](0,-2)--(-2,-.33){};
\draw[ thick,blue](0,-2)--(2,-2){};
\draw [>=stealth,->]   (2,-2)-- (3,-2);
\draw [>=stealth,->]   (-2,2)-- (-2,3);
\node at (0.5,-3) {\textcolor{magenta}{\textbf{Minimum occures in $M_1$}}};
\node at (3.3,-2) { $p_1$};
\node at (-2,3.25) { $p_2$};
\node at (1.6,-2.4) { $M_1$};
\node at (0.1,-2.4) { $M_2$};
\node at (2.3,-2.4) { $Q_1$};
\node at (-2.4,1) { $N_1$};
\node at (-2.4,-.33) { $N_2$};
\node at (-2.4,2) { $Q_2$};
\node at (2,2.4) { $(1,1)$};
\node at (0.25,-.8) { $h^*_{p_1}=0$};
\node at (-.8,-1.5) { $h^*_{p_2}=0$};
\node at (0,1.5) { $\frac{A_{21}}{n_1}>\frac{A_{12}}{n_2}$};

\draw[ very thick] (5,2) -- (5,-2)  (9,2) -- (9,-2)   (5,2) -- (9,2)   (5,-2)--(9,-2);
\draw[very thick,blue](8.5,-2)--(5,1){};
\draw[very thick,red](5,2)--(5,-.33){};
\draw[ thick,red](7,-2)--(5,-.33){};
\draw[ thick,blue](8.5,-2)--(9,-2){};
\draw [>=stealth,->]   (9,-2)-- (10,-2);
\draw [>=stealth,->]   (5,2)-- (5,3);
\node at (7.5,-3) {\textcolor{magenta}{\textbf{Minimum occures in $N_2$}}};
\node at (10.3,-2) { $p_1$};
\node at (5,3.25) { $p_2$};
\node at (8.6,-2.4) { $M_2$};
\node at (7.1,-2.4) { $M_1$};
\node at (9.3,-2.4) { $Q_1$};
\node at (4.6,1) { $N_2$};
\node at (4.6,-.33) { $N_1$};
\node at (4.6,2) { $Q_2$};
\node at (9,2.4) { $(1,1)$};
\node at (6,-1.3) { $h^*_{p_1}=0$};
\node at (6.1,0) { $h^*_{p_2}=0$};
\node at (7,1.5) { $\frac{A_{21}}{n_1}<\frac{A_{12}}{n_2}$};
\end{tikzpicture}
\end{center}

(ii) $\frac{nA_{12}}{n_2}<\frac{nA_{21}}{n_1}$  and   $\frac{1}{A_{21}}\ln \frac{nA_{21}}{n_1}<1$ we prove that the minimum of $h^*$ occurs in $M_1$.
For each point $(p_1,p_2)$ in unit square $[0,1]^2$ there is a unique point $(p'_1,p_2)$ on segments $M_1N_1$ or $N_1Q_2$ such that $h^*(p'_1,p_2)\le h^*(p_1,p_2)$ then the minimum of $h^*$ occurs on $M_1N_1\cup N_1Q_2$ also  there is a unique point $(p_1,p'_2)$ on segments $M_2N_2$ or $M_2Q_1$ such that $h^*(p_1,p'_2)\le h^*(p_1,p_2)$ then the minimum of $h^*$ occurs on $M_2N_2\cup M_2Q_1$ . This two sets of points intersect in one point, $M_1$ that is, $h(M_1)\le h^*(p_1,p_2)$ and  $$h^*(M_1)=h(\frac{1}{A_{21}}\ln \frac{nA_{21}}{n_1},0)=\frac{n_1}{A_{21}}(1+\ln \frac{nA_{21}}{n_1})$$

(iii) If $\frac{nA_{12}}{n_2}>\frac{nA_{21}}{n_1}$ and $\frac{1}{A_{12}}\ln \frac{nA_{12}}{n_2}<1$ we prove that the minimum of $h^*$ occurs in $N_2$.
For each point $(p_1,p_2)$ in unit square $[0,1]^2$ there is a unique point $(p'_1,p_2)$ on segments $M_1N_1$ or $N_1Q_2$ such that $h^*(p'_1,p_2)\le h^*(p_1,p_2)$, then the minimum of $h^*$ occurs on $M_1N_1\cup N_1Q_2$ also  there is a unique point $(p_1,p'_2)$ on segments $M_2N_2$ or $M_2Q_1$ such that $h^*(p_1,p'_2)\le h^*(p_1,p_2)$ then the minimum of $h^*$ occurs on $M_2N_2\cup M_2Q_1$. This two sets of points intersect in one point, $N_2$, that is, $h^*(N_2)\le h^*(p_1,p_2)$ and
$$h^*(N_2)=h^*(0,\frac{1}{A_{12}}\ln \frac{nA_{12}}{n_2})=\frac{n_2}{A_{12}}(1+\ln \frac{nA_{12}}{n_2}).$$

In each of three cases, if we set $T=\max\{\frac{nA_{12}}{n_2},\frac{nA_{21}}{n_1}\}$ then we have :
$$\underset{p_1,p_2} \min\  h^*(p_1,p_2)=n(\frac{1+\ln T}{T})$$
\end{proof}
we now pose a problem.\\
\textbf{Problem} 1. Minimize $h^*$ if $\delta_1=1$ or $\delta_2 =1$.

\subsubsection{ $k$ is odd}

We assume that $k$ is an odd integer and we wish to minimize $h^*(p_1,p_2)$. For this purpose, we use calculus methodes.
\[  \left\{ \begin{array}{ll} h_{p_1}=n_1-n_1(A_{11}+1)e^{-p_1(A_{11}+1)-p_2A_{12}} - n_2A_{21}e^{-p_1A_{21}-p_2(A_{22}+1)} \\  h_{p_2}=-n_1A_{12}e^{-p_1(A_{11}+1)-p_2A_{12}} + n_2-n_2(A_{22}+1)e^{-p_1A_{21}-p_2(A_{22}+1)} \end{array} \right.   \]
\[  \left\{ \begin{array}{ll} h_{p_1}=0 \\  h_{p_2}=0 \end{array} \right. \Longrightarrow \left\{ \begin{array}{ll} n_1(A_{11}+1)e^{-p_1(A_{11}+1)-p_2A_{12}} + n_2A_{21}e^{-p_1A_{21}-p_2(A_{22}+1)}=n_1 \\  n_1A_{12}e^{-p_1(A_{11}+1)-p_2A_{12}}+n_2(A_{22}+1)e^{-p_1A_{21}-p_2(A_{22}+1)}= n_2 \end{array} \right.  \]

Therefore we have:
\[\left\{ \begin{array}{ll}e^{-p_1A_{21}-p_2(A_{22}+1)}=\displaystyle \frac{n_2(A_{11}+1)-n_1A_{12}}{n_2[A_{12}A_{21}-(A_{11}+1)(A_{22}+1)]} \\e^{-p_1(A_{11}+1)-p_2A_{12}}=\displaystyle \frac{n_1(A_{22}+1)-n_2A_{21}}{n_1[A_{12}A_{21}-(A_{11}+1)(A_{22}+1)]} \end{array} \right.  \]
%
%
Let $E_1=\displaystyle \frac{n_2(A_{11}+1)-n_1A_{12}}{n_2[A_{12}A_{21}-(A_{11}+1)(A_{22}+1)]},\quad E_2=\displaystyle \frac{n_1(A_{22}+1)-n_2A_{21}}{n_1[A_{12}A_{21}-(A_{11}+1)(A_{22}+1)]}$.

If $E_1>0$ and $E_2>0$, then we have a linear equations system  $\left\{ \begin{array}{ll}p_1A_{21}+p_2(A_{22}+1)=-\ln E_1  \\  p_1(A_{11}+1)+p_2A_{12}=-\ln E_2 \end{array} \right.$ with a unique answer and we set : \[  \left\{ \begin{array}{ll}P_1=\displaystyle \frac{(A_{22}+1)\ln E_2-A_{12}\ln E_1}{A_{12}A_{21}-(A_{11}+1)(A_{22}+1)}  \\  P_2=\displaystyle\frac{(A_{11}+1)\ln E_1-A_{21}\ln E_2}{A_{12}A_{21}-(A_{11}+1)(A_{22}+1)} \end{array} \right.  \]

\begin{definition}
A connected bipartite graph $G$ is called \textit{$4$-perfect} if $E_1>0$ , $E_2>0$ where $E_1=\displaystyle \frac{n_2(A_{11}+1)-n_1A_{12}}{n_2[A_{12}A_{21}-(A_{11}+1)(A_{22}+1)]}$ and $E_2=\displaystyle \frac{n_1(A_{22}+1)-n_2A_{21}}{n_1[A_{12}A_{21}-(A_{11}+1)(A_{22}+1)]}$.
\end{definition}
We thus obtain the following.
\begin{corollary}\label{cro3}
If $G$ is a $4$-perfect graph and , $0< P_1< 1$ and $0< P_2< 1$, then \[  \underset{(p_1,p_2)\in(0,1)^2}{\min} h^*(p_1,p_2)=h^*(P_1,P_2)=n_1 [ E_2+P_1  ] + n_2 [ E_1 +P_2 ].   \]
\end{corollary}

Note that Corollary \ref{cro3} improve Theorem \ref{thm2.7} if $G$ is both perfect and $4$-perfect.

\begin{example}

 It remains to show that there are perfect graphs that are $4$-perfect as well.
For this purpose, we consider the graph introduced in Example  \ref{example}. Let $ n_1=n_2=\frac{n}{2}$, $\delta_1=\delta_2=\delta$, and $k=4m+1$. Then
 \[ E_1=E_2=\displaystyle \frac{1}{(2m+1)\delta+1} \text{ ,} P_1=P_2=\displaystyle \frac{\ln[(2m+1)\delta+1]}{(2m+1)\delta+1} \]
Since $E_1,E_2>0$, $G$ is $4$-perfect. It is also easy to see that $G$ is perfect.
\end{example}

\end{document}